\documentclass[10pt, english]{amsart}

\usepackage{amsmath,amssymb,enumerate}

\usepackage{mathrsfs}

\usepackage{todonotes} 
\usepackage[T1]{fontenc}
\usepackage[all,cmtip]{xy}

\usepackage{babel}
\usepackage{amstext}
\usepackage{amsfonts}
\usepackage{latexsym}
\usepackage{ifthen}
\usepackage{tikz-cd}
\usepackage{csquotes}

\makeatletter
\@namedef{subjclassname@2020}{\textup{2020} Mathematics Subject Classification}
\makeatother

\newcommand{\sK}{{\mathcal K}}

\newtheorem{lemma1}{}[section]

\newenvironment{lemma}{\begin{lemma1}{\bf Lemma.}}{\end{lemma1}}
\newenvironment{example}{\begin{lemma1}{\bf Example.}\rm}{\end{lemma1}}

\newenvironment{theorem}{\begin{lemma1}{\bf Theorem.}}{\end{lemma1}}

\newenvironment{proposition}{\begin{lemma1}{\bf Proposition.}}{\end{lemma1}}
\newenvironment{corollary}{\begin{lemma1}{\bf Corollary.}}{\end{lemma1}}
\newenvironment{remark}{\begin{lemma1}{\bf Remark.}\rm}{\end{lemma1}}

\newenvironment{definition}{\begin{lemma1}{\bf Definition.}}{\end{lemma1}}
\newenvironment{notation}{\begin{lemma1}{\bf Notation.}}{\end{lemma1}}

\newenvironment{setup}{\begin{lemma1}{\bf Setup.} \rm }{\end{lemma1}}

\newenvironment{conjecture}{\begin {lemma1}{\bf Conjecture.}}{\end{lemma1}}

\newenvironment{fact}{\begin{lemma1}{\bf Fact.}\rm}{\end{lemma1}}

\newenvironment{remark*}{{\bf Remark.}}{}
\newenvironment{remarks*}{{\bf Remarks.}}{}
\newenvironment{example*}{{\bf Example.}}{}
\newenvironment{assumption*}{{\bf Assumption.}}{}

\newcommand{\R}{\ensuremath{\mathbb{R}}}

\newcommand{\Z}{\ensuremath{\mathbb{Z}}}
\newcommand{\C}{\ensuremath{\mathbb{C}}}
\newcommand{\N}{\ensuremath{\mathbb{N}}}
\newcommand{\PP}{\ensuremath{\mathbb{P}}}

\usepackage{eurosym}

\newcommand{\holom}[3]{\ensuremath{#1:#2  \rightarrow #3}}
\newcommand{\fibre}[2]{\ensuremath{#1^{-1} (#2)}}

\makeatletter
\ifnum\@ptsize=0  \fi \ifnum\@ptsize=2  \fi 
\makeatother

\newcommand\sE{{\mathcal E}}

\newcommand\sF{{\mathcal F}}
\newcommand\sG{{\mathcal G}}

\newcommand\sO{{\mathcal O}}

\DeclareMathOperator*{\pic}{Pic}

\usepackage{xcolor}
\usepackage{hyperref}

\author{Francesco Antonio Denisi}
\author{Andreas H\"oring}

\address{Andreas H\"oring, Universit\'e C\^ote d'Azur, CNRS, LJAD, France}
\email{Andreas.Hoering@univ-cotedazur.fr}

\address{Francesco Antonio Denisi, Fachrichtung Mathematik, Campus, Gebäude E2 4, Universität des Saarlandes, 66123 Saarbrücken, Germany}
\email{denisi@math.uni-sb.de}


\title{Canonical extensions of ruled surfaces over elliptic curves} 
\date{October 2nd, 2026} 

\begin{document}

\begin{abstract}
Let $L$ be a line bundle of negative degree on an elliptic curve $E$. While  the ruled surface $M := \mathbb P(\sO_E \oplus L)$ is one of the simplest varieties in algebraic geometry, its tangent bundle $T_M$ is surprisingly complex: we first show that the ring of symmetric tensors is not finitely generated. We then construct a holomorphic function on the canonical extension $Z_M$ which allows us to show that $Z_M$ is not Stein.
\end{abstract}

\maketitle


\section{Introduction}

\subsection{Motivation and main result}

Given a projective manifold $M$ there are different ways to measure the positivity of its tangent bundle: from a differential-geometric point of view we can ask for the existence of metrics with (semi)positive (bi)sectional curvature (e.g. \cite{Mok88, Mat22}), while MMP suggest to study manifolds
such that the tautological bundle $\sO_{\PP(T_M)}(1)$   
on the Grothendieck projectivisation $\PP(T_M)$
is nef or pseudoeffective (e.g.\ \cite{DPS94, HLS22}). The most geometric measure is given by the
algebra of symmetric tensors
$$
R(M, T_M) := \bigoplus_{k \in \N} H^0(M, S^k T_M).
$$
Very recently Beauville and Liu \cite{BL26} studied the properties
of $R(M, T_M)$, they also gave an explicit  description for several classical varieties. For example if $M=\PP(\sE) \rightarrow C$ is a ruled surface
over a curve of genus at least two, the algebra $R(M, T_M)$ strongly depends on the stability properties of $\sE$ and its symmetric powers \cite{BL26}. 

The ring $R(M, T_M)$ is isomorphic to the ring of regular functions
on the total space of the cotangent bundle $\Omega_M$, a symplectic manifold that deforms to the 
{\it canonical extension $Z_M$} studied by Greb and Wong \cite{GW20}:
let $\omega \in H^1(M, \Omega_M) \simeq \mbox{Ext}^1(T_M, \sO_M)$ be a K\"ahler class. This defines an extension of vector bundles, 
$$
0 \to \sO_M \to V^{\omega} \to T_M \to 0 
$$
and therefore an embedding $\mathbb P(T_M) \subset \mathbb P(V^{\omega})$. Then 
$$
Z_M := Z_M^\omega :=   \mathbb P(V^{\omega}) \setminus \mathbb P(T_M)
$$
is an affine bundle over $M$ modelled on $\Omega_M$. 
It is expected that the geometry of the complex manifold $Z_M$ strongly reflects the positivity of the tangent bundle:

\begin{conjecture} \label{conjecture-GW} \cite{GW20} \cite[Conj.1.1]{HP24}
Let $M$ be a projective manifold, and let $Z_M$ be a canonical extension defined by some K\"ahler class on $M$. Then 
$Z_M$ is Stein if and only if $T_M$ is nef.
\end{conjecture}

Apart from some special cases \cite{GW20, HP24, HP24b, Mue25} this conjecture is wide open.
It is known for surfaces except for ruled surfaces over elliptic curves. This seemingly innocuous case is the object of this paper:

\begin{theorem} \label{thm:not-Stein}
Let $E$ be an elliptic curve, and $L$ a line bundle of negative degree over $E$. Set $M:=\PP(\sO_E\oplus L)\to E$. 
Consider the canonical extension 
$$
0 \to \sO_M \to V^{\omega} \to T_M \to 0
$$
associated with any K\"ahler class $\omega\in H^1(M,\Omega^1_M)$. Then the complement $\PP(V^{\omega})\setminus \PP(T_M)$ is not Stein.
\end{theorem}

Combined with \cite[Theorem 1.13]{HP24} and \cite[Lemma 0.3]{Mue25} this yields the following corollary.

\begin{corollary}
Conjecture \ref{conjecture-GW} holds if $\dim M=2$.
\end{corollary}

\subsection{Strategy of proof}

Let $E$ be an elliptic curve, and $L$ a line bundle of negative degree over $E$.  The ruled surface $M:=\PP(\sO_E\oplus L)\to E$
contains a unique curve $C_-$ with negative self-intersection, and it is not difficult to see that for any curve $C \subset M$
the restriction $T_M|_C$ is nef if and only if $C \neq C_-$. Thus $T_M$ is almost nef in the sense of \cite{DPS01} and we can expect the ring $R(M, T_M)$ to be non-trivial. In fact it is not difficult to show that $T_M$ and 
$V^\omega$ are big vector bundles (cf. Remark \ref{remark:extensionbig}). 
This suggests that one can use the geometric strategy developed in \cite{HP24b}:
construct a weak Fano model of the pair $(\PP(V^\omega), \PP(T_M))$, i.e.,  a sequence of anti-flips
$$
\psi: (\PP(V^\omega), \PP(T_M)) \dashrightarrow (\PP(V^\omega)^+, \PP(T_M)^+) 
$$
to a mildly singular variety $\PP(V^\omega)^+$ such that $\PP(T_M)^+ \subset \PP(V^\omega)^+$ is a nef and big divisor. The geometry of $\psi$ then allows to decide
whether $Z_M$ is Stein or not \cite[Prop.1.5]{HP24b}. However the existence of a weak Fano model is not clear and requires
the finite generation of $R(M, T_M)$. Our first result shows that this approach can not be used in our case:

\begin{theorem} \label{thm:not-finite}
Let $E$ be an elliptic curve, and $L$ a line bundle of negative degree over $E$. Set $M:=\PP(\sO_E\oplus L)\to E$. 
Then the tangent bundle $T_M$ is big and the ring of global sections
$$
R(M, T_M) = \bigoplus_{k \in \N} H^0(M, S^k T_M)
$$
is not finitely generated.
\end{theorem}

In summary we obtain an abundance of global sections of $\sO_{\PP(V^\omega)}(k)$ defining rational functions on $Z_M$,
but it is not clear how to control their properties in an abstract setting and use them to investigate the function theory of $Z_M$.
The key idea of this paper is to overcome this difficulty  by constructing directly a holomorphic function 
$$
\holom{\mu}{Z_M}{\C}
$$ 
without using
the embedding $Z_M \hookrightarrow \PP(V)$. 
For this purpose we show in Proposition \ref{prop:key} that if $L$ has degree $-1$,  the universal cover
$\tilde Z_M \rightarrow Z_M$ is isomorphic to a product $\C^2 \times Q$ where $Q$ is an affine quadric.
We also describe the action of the fundamental group via the factor of automorphy corresponding to the line bundle $L^*$. 
This preparation allows to define the function $\mu$ and find a special fibre that is not a Stein manifold
 (cf. Theorem \ref{thm:degree-one}). The case where the degree of $L$ is arbitrary then follows easily by using a suitable covering.

{\bf Acknowledgements.} This project started during a visit of the first-named author to Université Côte d’Azur. He would like to thank the Laboratoire J.\ A.\ Dieudonné, and the university for their hospitality and for providing a pleasant and stimulating working environment.
\medskip

A.H.\ was supported by the ANR-DFG project \enquote{Positivity on K-trivial varieties}, ANR-23-CE40-0026 and DFG Project-ID 530132094, the ANR-Project ``Groups in algebraic geometry'' ANR-24-CE40-3526, and partially supported by the France 2030 investment plan managed by the ANR, as part of the Initiative of Excellence of Universit\'e C\^ote d'Azur, reference ANR-15-IDEX-01. F.\ A.\ D.\ was supported by the by the European Research Council (ERC) under the European Union's Horizon 2020 research and innovation program (ERC-2020-SyG-854361-HyperK), and by the Deutsche Forschungsgemeinschaft (DFG, German Research Foundation) – Project ID 286237555 (TRR 195).

\section{Notation and basic facts}

We work over the complex numbers, for general definitions we refer to \cite{Har77, Kau83}. 
Complex spaces and varieties will always be supposed to be irreducible and reduced. 
We follow \cite{Laz04b} for notions of positivity of vector bundles. We will use interchangeably the notions of line bundles, divisors and invertible sheaves on a complex manifold.

\subsection{Reminder on vector bundles and their extensions} \label{subsection:reminder-cech}

Given a vector bundle $\sE$ on a complex manifold $X$, we denote its total space by $\mathrm{Tot}(\sE)$. We denote by 
$$
\mathbf P(\sO_X \oplus \sE) \rightarrow X\; (\text{resp.\ } \PP(\sO_X\oplus \sE)\to X)
$$
the {\em non-Grothendieck} (resp.\ {\em Grothendieck}) projectivisation of the vector bundle $\sO_X \oplus \sE$. We can define an embedding
$$
 \mathrm{Tot}(\sE) \hookrightarrow  \mathbf P(\sO_X \oplus \sE), \qquad (x,s) \ \mapsto (x, [1:s]);
$$
one checks that this local formula is well-defined and that the image is
$$
\mathbf P(\sO_X \oplus \sE) \setminus  \mathbf P(0 \oplus \sE).
$$
Since $\mathbf P(\sO_X \oplus \sE) \simeq \PP(\sO_X \oplus \sE^*)$, we have an embedding
$$
\mathrm{Tot}(\sE) \hookrightarrow \PP(\sO_X \oplus \sE^*)
$$
whose image is $\PP(\sO_X \oplus \sE^*) \setminus \PP(\sE^*)$.

\begin{example} \label{example:complement}
Let $X$ be a smooth projective curve, and let $L^*$ be a line bundle of degree $1$ on $X$.
Then we have an inclusion
$$
\mathrm{Tot}(L^*) \hookrightarrow \PP(\sO_X \oplus L) =: M
$$
whose complement is the curve $C_-:=\PP(L) \subset \PP(\sO_X \oplus L)$. 
Since $C_-^2=-1$ we can use Mumford's criterion to obtain a birational morphism
$$
\PP(\sO_X \oplus L) \rightarrow M'
$$
contracting $C_-$ onto a point $p$. The embedding 
$$
\mathrm{Tot}(L^*) \hookrightarrow M'
$$
is thus a one-point compactification of the total space. 
Let $C_+ \subset \mathrm{Tot}(L^*)$ be the zero section, then 
by the above we have an embedding
$$
\mathrm{Tot}(L^*) \setminus C_+ \hookrightarrow M'  \setminus C_+. 
$$ 
Since the complement is just the point $p$, we know by \cite[V, Theorem 4]{GR79} that $\mathrm{Tot}(L^*) \setminus C_+$ is not Stein.
\end{example}

\begin{definition}
Let $\sE$ be a vector bundle on a projective manifold, and denote by $\PP(\sE)$ the Grothendieck projectivisation.
We see that the vector bundle $\sE$ is ample (resp. nef, big, pseudoeffective, globally generated, semiample) if the tautological
line bundle $\sO_{\PP(\sE)}(1)$ is ample (resp. nef, big, pseudoeffective, globally generated, semiample).
\end{definition}

\begin{remark*}
While the definition of ample and nef vector bundles is uniform in the literature, this is not the case for big, pseudoeffective or semiample (cf. \cite{Bplus15, MU19} for an overview). 
\end{remark*}

\begin{remark} \label{remark:nef-big}
It can be difficult to verify that a vector bundle is big, this becomes much easier if the vector bundle is nef:
let $\sF \rightarrow M$ be a nef vector bundle of rank $r$ on a projective surface, and denote by $\sO_{\PP(\sF)}(1)$ the tautological bundle. Then $\sO_{\PP(\sF)}(1)$ is big if and only if
$$
(c_1(\sO_{\PP(\sF)}(1)))^{r+1}>0.
$$
By the definition of Segre classes this is equivalent to showing that
$$
c_1(\sF)^2-c_2(\sF)>0.
$$
\end{remark}

\begin{fact} \label{fact-cech}
Consider an extension of locally free sheaves
\begin{equation}\label{eqn:extension}
0 \to \sF \to \sE \to \sO_X \to 0
\end{equation}
on a complex space $X$, and recall \cite[III, Ex.6.1]{Har77} that the corresponding extension
in $\mbox{Ext}^1(\sO_X, \sF) \simeq H^1(X, \sF)$ is the image of $1 \in H^0(X, \sO_X)$
under the connecting homomorphism
$$
\delta\colon H^0(X, \sO_X) \rightarrow H^1(X, \sF).
$$
Moreover, fix an open cover $\mathcal U=\{U_i\}$ on which $1\in\sO(U_i)$ admits a lift $\sigma_i\in\sE(U_i)$. On $U_i\cap U_j$ we set
$$
c_{ij}:=\sigma_j-\sigma_i\in\sF(U_i\cap U_j).
$$ 
Then $\delta(1) = \{c_{ij}\}$, so it is a \v{C}ech $1$-cocycle representing the extension class of \ref{eqn:extension}.
If $q_i\in\sF(U_i)$ is a local section of the kernel, then $\sigma_i-q_i$ is again a lift of $1$, and the cocycle defined by these new lifts is $c_{ij}-(q_j-q_i)$. 
\end{fact}

\begin{example} \label{example:Serre}
Let $E$ be an elliptic curve. Then $\mbox{Ext}^1(\sO_E, \sO_E) = H^1(E, \sO_E) = \C$, so up to isomorphism there exists a unique non-trivial extension
$$
0 \rightarrow \sO_E \rightarrow \mathscr F_2 \rightarrow  \sO_E \rightarrow 0.
$$
We call $\mathscr F_2$ the Serre bundle, it is nef but not semiample \cite[Example 1.7]{DPS94}: we have
$$
H^0(\PP(\mathscr F_2), \sO_{\PP(\mathscr F_2)}(k)) = \C \qquad \forall \ k \in \N.
$$
The tautological bundle $\sO_{\PP(\mathscr F_2)}(1)$ is nef and not big, so we have
$$
\mbox{Nef}(\PP(\mathscr F_2)) = \R_{\geq 0} \sO_{\PP(\mathscr F_2)}(1) + \R_{\geq 0} \sO_{\PP(\mathscr F_2)}(l)
$$
where $l$ is a fibre of the ruling. Since $\mathscr F_2$ is semistable we know by \cite[Section.1.5.A]{Laz04a} that
$$
\mbox{Nef}(\PP(\mathscr F_2)) = \mbox{Pseff}(\PP(\mathscr F_2)). 
$$
\end{example}


\subsection{Factors of automorphy}

Let us give a summary of a special case of factors of automorphy as presented in \cite[Appendix B]{BL04}, cf. \cite{Ien11}
for a detailed exposition.

\begin{definition}
Let $E = \C/\Lambda$ be an elliptic curve, and identify the lattice $\Lambda$ with $\pi_1(E)$. To avoid confusion, we will denote the universal cover $\C$ by $\tilde E$.
A factor of automorphy is a holomorphic function
$$
f \colon \Lambda \times \tilde E \rightarrow \C^*
$$
such that, for every $\alpha, \beta \in \Lambda$, and every point $z \in \tilde E$ we have
$$
f(\alpha+\beta, z) = f(\alpha, z+\beta) \cdot f(\beta, z).
$$
\end{definition}

\begin{definition} \label{defn:automorphy-line-bundle}
Let $E = \C/\Lambda$ be an elliptic curve, and 
$$
f \colon \Lambda \times \tilde E \rightarrow \C^*
$$
a factor of automorphy. Then we have a free and properly discontinuous group action of $\Lambda$ on $\tilde E \times \C$
given by
$$
\lambda\cdot(z,t) := (z+\lambda, f(\lambda, z) \cdot t),
$$
and the quotient $(\tilde E \times \C)/\Lambda$ is the total space of a line bundle over $E$.
\end{definition}

\begin{example} \label{example:automorphy} \cite[Ex.4.2]{Ien11}
Let $\Lambda = \Z \oplus \Z \tau$ be a lattice and set $E:=\C/\Lambda$.
Fix a point  $p \in E$ and write $p=\overline{a \tau+ b}$ with $a, b \in \R$  
The function
$$
b \colon \Lambda \times \tilde E \rightarrow \C^*, \qquad
(\lambda = m \tau+n, z) \ \mapsto \ e^{i \pi \left(2 a \lambda - m^2 \tau- 2m (z+a \tau+ b)\right)}
$$
is a factor of automorphy defining the line bundle $\sO_E(p)$ on the elliptic curve $E$.  
For fixed $\lambda \in \Lambda$ we will denote
$$
b_\lambda \colon \tilde E \rightarrow \C^*, \qquad z \ \mapsto \ b(\lambda, z).
$$
A straightforward computation shows that for $\lambda=m \tau+n$ one has
$$
\frac{b_\lambda'(z)}{b_\lambda(z)}=-2\pi i m.
$$
\end{example}

\subsection{Coordinates on an affine quadric}

\begin{notation} \label{notationQ}
Let $[x_0:x_1]$ and $[w_0:w_1]$ be homogeneous coordinates on the two factors of $\PP^1\times\PP^1$, and set for $p \in \PP^1 \times \PP^1$
$$
d_{Q}(p):=x_0 w_1- x_1 w_0.
$$ 
We set 
$$
Q:= \{ p \in \PP^1 \times \PP^1 \ | \  d_{Q}(p) \neq 0\}.
$$ 
\end{notation}

We note that $Q$ is an affine quadric; it is the complement of a hyperplane section for the Segre embedding
$\PP^1 \times \PP^1 \hookrightarrow \PP^3$.

\begin{lemma} \label{lemma:coordinates-quadric}
With the notation above, set 
$$
Q_0:=Q\cap\{x_0\neq0\}, \qquad Q_\infty:=Q\cap\{x_1\neq0\},
$$
and fix $0 \neq \rho \in \C$.
Then 
\[
Q_0 \to\C_t\times\C_y,\; ([x_0:x_1],[w_0:w_1]) \mapsto\left(\frac{x_1}{x_0},\frac{\rho x_0w_0}{x_0 w_1- x_1 w_0}\right)
\]
\[
Q_\infty\to\C_v\times\C_{y_\infty},\; ([x_0:x_1],[w_0:w_1])\mapsto\left(\frac{x_0}{x_1},\frac{- \rho x_1w_1}{x_0 w_1- x_1 w_0}\right)
\]
are affine coordinate charts with inverse
\[
\psi_0\colon \C_t\times\C_y\to Q_0,\; (t,y)\mapsto([1:t],[y:ty+\rho]).
\]
\[
\psi_\infty\colon \C_v\times\C_{y_\infty}\to Q_\infty,\; (v,y_\infty)\mapsto([v:1],[vy_\infty+\rho:y_\infty]).
\]
On the overlap $Q_0\cap Q_\infty$ we have the relations
\[
v=\frac{1}{t},\; y_\infty=-t^2y-\rho t.
\]
\end{lemma}

\begin{proof}
This is elementary: note that the equalities
$$
(d_Q \circ \psi_0)(t,y)=\rho, \qquad (d_Q \circ \psi_\infty)(v,y_\infty)=-\rho
$$
make it easy to check that the maps are indeed inverses of each other.
\end{proof}

\section{Geometry of the ruled surface} \label{section:geometry}

\begin{setup} \label{setup:ruled}
Let $E$ be an elliptic curve, and let $L$ be a line bundle of degree $-d<0$ on $E$. We set
$$
f \colon M:= \PP(\sO_E \oplus L) \rightarrow E.
$$
We denote by $\PP(L) =: C_- \subset M$ the negative section and note that $C_- \simeq \sO_{\PP(\sO_E \oplus L)}(1)$.
The quotient $\sO_E \oplus L \rightarrow \sO_E$ defines a curve $C_+ \subset M$, it is elementary to see
that $C_+$ is an element of the linear system $|C_- + f^* L^*|$. 
By construction the curves $C_-$ and $C_+$ are disjoint, so we have  $C_+ \cdot C_-=0$
\end{setup}

From the formula for the canonical bundle of a projectivised bundle we obtain 
\begin{equation}
\label{anticanonicalM}
- K_M \simeq 2 C_- + f^* L^* \simeq C_- + C_+.
\end{equation}

\begin{lemma} \label{lemma-cohomologyKM}
In the Setup \ref{setup:ruled} one has
\begin{equation}
h^0(M, -K_M) = d+1 = h^1(M, -K_M)
\end{equation}
and $h^2(M, -K_M)=0$.
\end{lemma}

\begin{proof}
We have $-K_M \simeq f^* L^* + 2C_-$ and therefore 
$$
f_* \sO_M(-K_M) \simeq L^* \otimes S^2 (\sO_E \oplus L).
$$ 
The statement now follows from the Leray spectral sequence and the Riemann-Roch formula on $E$.
\end{proof}

The curve  $C_- \subset M$ being a section of the fibration $f$, its normal sequence is split, i.e., we have
$$
T_{M}|_{C_-} \simeq \sO_{C_-} \oplus L_-
$$
where $L_- := N_{C_-/M} \simeq L$. Since $L_-$ has negative degree, we have a unique surjective morphism
\begin{equation}
\label{negative-quotient}
\kappa \colon T_M|_{C_-} \twoheadrightarrow L_-
\end{equation}
and its kernel is the subbundle $T_{C_-} \subset T_M|_{C_-}$.
Consider the exact sequence
$$
0 \rightarrow -K_M \rightarrow T_M \rightarrow f^* T_E \simeq \sO_M \rightarrow 0.
$$
This sequence is not split: if $T_M \simeq -K_M \oplus \sO_M$, it is generated by global sections on the complement of $C_-$.
Thus $M$ is almost homogeneous, but by \cite[Theorem 3]{Potters69} this implies $\deg L=0$, a contradiction. Thus we have 
\begin{equation}
\label{sectionsTM}
 H^0(M, T_M)  \cong H^0(M, -K_M) \cong \C^{d+1}.
\end{equation}

Let $\sF \rightarrow M$ be the kernel of the morphism
$$
T_M \twoheadrightarrow T_{M}|_{C_-} \stackrel{\kappa}{\twoheadrightarrow} L_-,  
$$
so we have an exact sequence
\begin{equation}
\label{eqn-defineF}
0 \rightarrow \sF \rightarrow T_M \rightarrow L_- \rightarrow 0
\end{equation}
Then $\sF$ is locally free of rank two  \cite[Theorem 1.4]{Mar82} and the following holds:

\begin{lemma}\label{lemma:defineF}
In the Setup \ref{setup:ruled}  the vector bundle $\sF$ is given by an extension
\begin{equation}\label{eqn:extensionF}
0 \rightarrow \sO_M(C_+) \rightarrow \sF \rightarrow \sO_M \rightarrow 0
\end{equation}
In particular, $\sF$ is big and nef.
\end{lemma}

\begin{proof}
Note first that the composition
$$
\sF \rightarrow T_M \rightarrow f^* T_E \simeq \sO_M
$$
is surjective: this is clear on the complement of $C_-$, and on the curve $C_-$ the image of
$$
\sF|_{C_-} \rightarrow T_M|_{C_-} 
$$
is by definition of $\kappa$ equal to the subbundle $T_{C_-} \subset T_M \otimes \sO_{C_-}$.
Yet the tangent map maps $T_{C_-}$ surjectively onto $f^* T_E \otimes \sO_{C_-}$,
so $\sF \rightarrow \sO_M$ is surjective.

Thus we have an exact sequence
$$
0 \rightarrow \sK \rightarrow \sF \rightarrow \sO_M \rightarrow 0
$$
and we are left to determine the line bundle $\sK \simeq \det \sF$.
The morphism $\sF \rightarrow T_M$ induces an injective morphism
$
\det \sF \rightarrow \det T_M \simeq -K_M
$
that vanishes with order one on $C_-$. Thus $\det \sF + C_- \simeq -K_M$, combined with  \eqref{anticanonicalM}
this yields $\det \sF \simeq C_+$.

For the second statement, note first
that $\sF$ is an extension of nef line bundles, so it is nef. Moreover $c_1^2(\sF)=C_+^2>0$ and $c_2(\sF)=0$, so $c_1^2(\sF)-c_2(\sF)>0$. Thus $\sF$ is nef and big by Remark \ref{remark:nef-big}.
\end{proof}

\begin{lemma}\label{lem:cohomologyF}
In the Setup \ref{setup:ruled}  we have
$$
h^0(M,\sF)=1+d, \qquad h^1(M,\sF)=1,
$$
and $h^i(M,\sF)=0$ for $i\geq 2$.
\end{lemma}
\begin{proof}
By \eqref{sectionsTM} and Lemma \ref{lemma-cohomologyKM} we know that 
$$
h^0(M,T_M)=h^0(M,-K_M)=h^0(M, C_+)=1+d
$$ 
Moreover since $\deg L_- = \deg L<0$ it follows from \eqref{eqn-defineF}
that $h^0(M,T_M)=h^0(M,\sF)$ so $h^0(M,\sF)=1+\mathrm{deg}(L^*_-)$. Pushing forward the exact sequence
\eqref{eqn:extensionF}
we obtain an extension
$$
0 \rightarrow L^* \oplus \sO_E \rightarrow f_* \sF \rightarrow \sO_E \rightarrow 0
$$
and $R^j f_* \sF=0$ for all $j \geq 1$. In particular we have $H^k(M, \sF) = H^k(E, f_* \sF)$.
Moreover, we have $c_1(f_* \sF)=d$ and therefore by Riemann-Roch
$$
h^1(E, f_* \sF) = h^0(E, f_* \sF) - \chi(E, f_* \sF) = 1+d - d = 1.
$$
\end{proof}

\begin{proposition}\label{prop:F_CisSerre}
In the Setup \ref{setup:ruled}   the restriction $\sF|_{C_-}$ is isomorphic to the Serre bundle $\mathscr{F}_2$ (cf.
Example \ref{example:Serre}).
In particular, $\sF$ is nef and big, but not semiample.
\end{proposition}

\begin{proof}
Restricting \eqref{eqn:extensionF} to $C_-$ we obtain an extension
$$
0\to \sO_{C_-}  \to  \sF|_{C_-} \to \sO_{C_-}  \to 0.
$$
Then either $\sF_{|C_-}\cong  \sO_{C_-} ^{\oplus 2}$, or $\sF_{|C_-} \cong \mathscr{F}_2$ and we are in
the second case if and only if $h^0(C_-, \sF|_{C_-} )=1$. 

Twisting  \eqref{eqn:extensionF}  with $\sO_M(-C_-)$ we obtain
$$
0 \rightarrow \sO_M(C_+ -C_-) \simeq f^* L^* \rightarrow \sF \otimes \sO_M(-C_-) \rightarrow \sO_M(-C_-) \rightarrow 0
$$
and compute that $h^0(M, \sF \otimes \sO_M(-C_-))=h^0(E, L^*)=d$
and $h^1(M, \sF \otimes \sO_M(-C_-))=h^1(E, L^*)=0$. 
Now consider the exact sequence
$$
0 \rightarrow \sF \otimes \sO_M(-C_-) \rightarrow \sF \rightarrow \sF|_{C_-} \rightarrow 0.
$$
Then the long exact sequence in cohomology and Lemma \ref{lem:cohomologyF} yield
the first statement.
For the second statement, note that the restriction of a semiample vector bundle
to any subvariety is semiample. Yet the Serre bundle is not semiample (cf.
Example \ref{example:Serre}).
\end{proof}

We denote by $\pi\colon \PP(T_M) \to M$ the projectivisation of the tangent bundle, and by 
$$
C_0 := \PP(L_-) \subset \PP(T_M)
$$
the curve corresponding to the quotient line bundle \eqref{negative-quotient}.
Let $\zeta:=c_1(\sO_{\PP(T_M)}(1))$ be the tautological class on $\PP(T_M)$.

\begin{lemma}
In the Setup \ref{setup:ruled},
let $N \subset \PP(T_M)$ be a curve such that $\zeta \cdot N<0$. Then $N=C_0$.
\end{lemma}

\begin{proof}
Since $\sF$ is nef by Lemma \ref{lemma:defineF}, it is clear from the exact sequence
\eqref{eqn-defineF} that $\pi(N)=C_-$. Yet $T_M \otimes \sO_{C_-} \simeq \sO_{C_-} \oplus L_-$,
so $C_0$ is the unique $\zeta$-negative curve mapping onto $C_-$.
\end{proof}

The exact sequence \eqref{eqn-defineF} allows us to see the vector bundle $\sF$ as an elementary transform (in the sense of Maruyama \cite[Theorem 1.4]{Mar82}) of $T_M$: this means that we have a commutative diagram
\begin{equation} \label{diagram-transform}
\xymatrix{
& Y \ar[rd]^{\nu} \ar[ld]_\tau &
\\
\PP(\sF) \ar[d] & & \PP(T_M) \ar[d]^\pi
\\
M \ar[rr] & & M
}
\end{equation}
where $\holom{\nu}{Y}{\PP(T_M)}$ is the blow-up along the curve $C_0$ and
$\holom{\tau}{Y}{\PP(\sF)}$ is the blow-down of the strict transform of $\PP(T_M|_C)$ onto a curve.
Denote by $F$ the exceptional divisor of $\nu$, then one has
$$
\nu^* \sO_{\PP(T_M)}(1)  \simeq \tau^* \sO_{\PP(\sF)}(1) + F
$$
and $\tau$ maps the exceptional divisor $F$ isomorphically onto $\PP(\sF \otimes \sO_{C_-})$.

\begin{lemma}\label{lem:normalC_0}
In the Setup \ref{setup:ruled},
we have $N^*_{C_0/\PP(T_M)}\cong  L^*\otimes \mathscr{F}_2$.
\end{lemma}

\begin{proof}
Set $S:=\pi^* C_-$, then 
we have $C_0 \subset S \subset \PP(T_M)$ and therefore an exact sequence 
\begin{equation}\label{eqn:normalx3}
0\to N^*_{S/\PP(T_M)} \otimes \sO_{C_0} \to  N^*_{C_0/\PP(T_M)} \to N^*_{C_0/S} \to 0,
\end{equation}
which reads 
$$
0\to L^* \to N^*_{C_0/\PP(T_M)} \to L^* \to 0.
$$ 
Yet 
$$
\PP(N^*_{C_0/\PP(T_M)}) \cong F \cong \PP(\sF \otimes \sO_{C_-}) \cong \PP(\mathscr{F}_2)
$$
where in the last step we used Proposition \ref{prop:F_CisSerre}.
Thus \eqref{eqn:normalx3} is unsplit and we obtain the statement.
\end{proof}

We can now prove our first main result:

\begin{proof}[Proof of Theorem \ref{thm:not-finite}]
We use the notation of the Setup \ref{setup:ruled} and the maps summarised in Diagram \eqref{diagram-transform}.
The vector bundle $T_M$ is big since it contains the subsheaf $\sF$ which is nef and big by Lemma \ref{lemma:defineF}.
We claim that
$$
\nu^* \sO_{\PP(T_M)}(1) \simeq \tau^* \sO_{\PP(\sF)}(1) + F
$$
is the divisorial Zariski decomposition of the big divisor $\nu^* \sO_{\PP(T_M)}(1)$. Assuming this for the time being, let us see how to conclude: by \cite[Theorem 5.5]{Bou04}
$$
H^0(Y, \nu^* \sO_{\PP(T_M)}(k))
\simeq H^0(Y, \tau^* \sO_{\PP(\sF)}(k))
$$
for every $k \in \N$. Thus we have isomorphisms
\begin{multline}
R(M, T_M) \simeq R(\PP(T_M), \sO_{\PP(T_M)}(1))
\simeq R(Y, \nu^* \sO_{\PP(T_M)}(1))
\\
\simeq R(Y, \tau^* \sO_{\PP(\sF)}(1))
\simeq R(\PP(\sF), \sO_{\PP(\sF)}(1)).
\end{multline}

The line bundle $\sO_{\PP(\sF)}(1)$ is nef and big but not semiample by Proposition \ref{prop:F_CisSerre}. Thus its ring of global sections is not finitely generated
by \cite[Theorem 2.3.15]{Laz04a}.

{\em Proof of the claim.} 
Let $\holom{q}{\PP(N_{C_0/\PP(T_M)}^*)}{C_0}$ be the ruling given by the blow-up~$\nu$. Since $(\nu^* \sO_{\PP(T_M)}(1))|_{F} \simeq q^* L$ is not pseudoeffective, it is clear
that $F$ is in the negative part of the divisorial Zariski decomposition of $\nu^* \sO_{\PP(T_M)}(1)$. Thus we can write
$$
\nu^* \sO_{\PP(T_M)}(1) = P + \lambda F
$$
where $\lambda>0$ and $P$ is a big class such that $P|_{F}$ is pseudoeffective.

Since $N_{C_0/\PP(T_M)}^* \simeq L^* \otimes \mathscr{F}_2$, by Proposition \ref{prop:F_CisSerre} 
we have
$$
\sO_{F}(-F) \simeq \sO_{\PP(N_{C_0/\PP(T_M)}^*)}(1) \simeq \sO_{\PP(\mathscr{F}_2)}(1) \otimes q^* L^*.
$$
Let us now show that $\lambda = 1$.
We have
$$
(\nu^* \sO_{\PP(T_M)}(1) - \lambda F)|_{F}
=
\lambda \sO_{\PP(\mathscr{F}_2)}(1) +  (1-\lambda) q^* L,
$$
and the pseudoeffective cone of $\PP(\mathscr{F}_2)$ is generated by $\sO_{\PP(\mathscr{F}_2)}(1)$
and a fibre of the ruling, so the divisor above is pseudoeffective if and only if
$1-\lambda \leq 0$. Thus the coefficient of $F$ in the negative part of $\nu^* \sO_{\PP(T_M)}(1)$ is $\lambda \geq 1$, and therefore
$$
H^0(Y, \nu^* \sO_{\PP(T_M)}(k))
=
H^0(Y, \nu^* \sO_{\PP(T_M)}(k)-kF)
$$
for every $k \in \N$. But $\nu^* \sO_{\PP(T_M)}(1)-F=\tau^* \sO_{\PP(\sF)}(1)$ is nef
and big, thus by \cite[Theorem 5.5]{Bou04} the divisor $\nu^* \sO_{\PP(T_M)}(1)-F$ is the positive part of the divisorial Zariski decomposition of $\nu^* \sO_{\PP(T_M)}(1)$.
\end{proof}

\begin{remark} \label{remark:extensionbig}
For completeness sake let us indicate how to extend the results above to the canonical extension 
$$
0 \rightarrow \sO_M \rightarrow V^\omega \rightarrow T_M \rightarrow 0
$$
associated to some K\"ahler class $\omega$ on $M$:
let $\beta$ be the composite morphism $V^\omega\to T_M \to L_-$, and set $\sG:=\mathrm{ker}(\beta)$. By \cite[Theorem 1.4]{Mar82} the sheaf $\sG$ is locally free of rank 3, by construction we obtain a commutative diagram

\begin{equation}\label{bigdiagram}
\centering
\begin{tikzcd} 
       &                                    & 0\arrow [d]        &  0\arrow [d]       & \\
0 \rar & \sO_M \rar \arrow[d,equal] & \sG \rar \arrow[d] & \sF \rar \arrow[d] & 0 \\
0 \rar & \sO_M \rar & V^\omega \rar \arrow[d] \arrow[dr, dashed, "\beta"] & T_M \rar \arrow{d} & 0 \\
       &            &              L_- \arrow[r,equal] \arrow[d]     & L_-    \arrow[d]             &          \\
       &            &                0   & 0                            &.
\end{tikzcd}
\end{equation}

The vector bundle $\sG$ is nef and big since it is an extension of the nef and big vector bundle $\sF$ by $\sO_M$.
The tautological bundle of $\sG$ is not semiample since its restriction to 
$\PP(\sF|_{C-}) \simeq \PP(\mathscr F_2)$ is not semiample. 
The inclusion $\sG \subset V^\omega$ implies that $V^\omega$ is big.
Arguing as in the proof of Theorem \ref{thm:not-finite} one obtains
that
$$
R(M, V^\omega) :=  \bigoplus_{k \in \N} H^0(M, S^k V^\omega)
$$
is not finitely generated. Since we will not use these results we leave the proof to the interested reader.
\end{remark}

\section{Universal cover of the canonical extension} \label{section:universal}

\begin{setup} \label{setup1} 
Let $\Lambda = \Z \oplus \Z \tau \subset \C$ be a lattice, and let
$E := \C/\Lambda$ the elliptic curve. 
Let $L$ be a line bundle of degree $-1$ on $E$, so $L^* \simeq \sO_E(p)$ for some point $p \in E$.
We keep the notation 
$$
f\colon M :=\PP(\sO_E\oplus L)\to E,
$$
and  $C_-:=\PP(L), C_+:=\PP(\sO_E)$ introduced in Section \ref{section:geometry}.
We fix a K\"ahler class $\omega$ on $M$ and denote by
\begin{equation}
\label{eqn-canonicalext}
0 \rightarrow \sO_M \rightarrow V^\omega \rightarrow T_M \rightarrow 0
\end{equation}
the corresponding extension. For simplicity of notation we set
$$
Z := Z_M^\omega := \PP(V^\omega)\setminus\PP(T_M)
$$
and denote by
 $$
 \holom{\pi}{Z}{M}
 $$ 
 the restriction of $\holom{\pi}{\PP(V^{\omega})}{M}$ to $Z$.

Let $\holom{p}{\C}{E}$ be the universal cover. Taking the fibre products, we obtain a commutative diagram
 \begin{equation}\label{diagram-universal}
\centering
\begin{tikzcd} 
\tilde Z \arrow[r, "p_Z"] \arrow[d, "\tilde \pi"] & Z \arrow[d, "\pi"]
\\
\tilde M \arrow[r, "p_M"] \arrow[d, "\tilde f"] & M \arrow[d, "f"]
\\
\tilde E \simeq \C \arrow[r, "p"]  & E
\end{tikzcd}
\end{equation}

Since $p_Z$ and $p_M$ are the universal covers, the group $\Lambda$ acts on $\tilde Z$ and $\tilde M$ in a way that is equivariant for the maps in the diagram.
\medskip

In what follows, we will denote by $m \tau + n$ with  $m, n \in \Z$ the elements of $\Lambda$. 
We denote by $z \in \C$ the coordinate on $\tilde E \simeq \C$. For clarity's sake
we will denote $\C_z := \tilde E$.

Finally let $- \xi_C\in H^1(M,\Omega_M)$ be the cohomology class associated with the curve $C_-$\footnote{The unusual choice of notation is convenient for our computations.}. 
Since $H^1(M,\Omega_M) = f^* H^1(E, \Omega_E) \oplus \R \xi_C$ we can write
\begin{equation}
\label{decompose-omega}
\omega :=\rho \xi_C+ f^*\beta_\omega,
\end{equation}
where $\rho<0$ and $\beta_\omega \in H^1(E,\Omega_E^1)$. 
Recall that $H^1(E,\Omega_E^1) \simeq \C$, and denote by $K \in \C$ the complex number corresponding to $\beta_\omega$. 
Since $\rho \xi_C$ is a positive multiple of the class of the negative section $C_-$ and $\omega$ is a K\"ahler class we have $K \neq 0$.
\end{setup}

The goal of this section is to prove the following proposition.

\begin{proposition} \label{prop:key}
We have an isomorphism 
\[
F\colon\widetilde Z\xrightarrow{\sim}\C_z\times\C_u\times Q.
\]
Moreover the $\Lambda$-action on $\C_z\times\C_u\times Q$ is given  by
\begin{align*}
\lambda &\cdot (z,u,[x_0:x_1],[w_0: w_1])\\
&=\Bigg(z+\lambda, u+m\Bigg(K+2\pi i \rho \frac{x_1 w_0}{x_0w_1-x_1w_0}\Bigg), [x_0: b_\lambda(z) x_1],[w_0: b_\lambda(z) w_1]\Bigg).
\end{align*}
\end{proposition}

The proof requires some preliminary work.  

\begin{lemma}
We have $\widetilde M\simeq \C_z\times\PP^1$ and, in the coordinates $(z,[x_0:x_1])$, the $\Lambda$-action is given by
\begin{equation}
\label{action:tildeM}
\lambda\cdot(z,[x_0:x_1])=(z+\lambda,[x_0:b_\lambda(z)x_1]),
\end{equation}
for every $\lambda\in\Lambda$ and $b_\lambda(z)$ the factor of automorphy defined in Example \ref{example:automorphy}.
\end{lemma}

\begin{proof}
We have $\widetilde M=M\times_E\C\simeq\PP(\sO_{\C}\oplus p^*L)$, and since $p^*L$ is trivial, it follows that $\widetilde M\simeq\C_z\times\PP^1$. We now describe the $\Lambda$-action on $\widetilde M$. First of all, we recall that
$$
\PP(\sO_E \oplus L) = \mathbf P(\sO_E \oplus L^*)
$$
where $\mathbf P(\sO_E \oplus L^*)$ is the projectivised bundle of lines in the total space $\sO_E \oplus L^*$.
Since by definition $\sO_E \oplus L^* \simeq \sO_E \oplus \sO_E(\bar p)$ it follows
from Definition \ref{defn:automorphy-line-bundle} that  
$$
\lambda\cdot(z,x_0,x_1)=(z+\lambda,x_0,b_\lambda(z)x_1)
$$
gives an action on $\C_z \times \C_{x_0} \times \C_{x_1}$ such that the quotient is the total space of $\sO_E \oplus L^*$. The action \eqref{action:tildeM} is the same action in homogeneous coordinates.
\end{proof}

For later use, we need to write the action \eqref{action:tildeM} in affine charts:

\begin{notation} \label{notation:tildeM}{\rm
We choose homogeneous coordinates $[\alpha:\beta]$ on the $\PP^1$-factor  
so that $p_M^{-1}(C_-)=\{\alpha=0\}$ and $p_M^{-1}(C_+)=\{\beta=0\}$. 
In particular, we have
\[
\widetilde M_0:=\widetilde M\setminus p_M^{-1}(C_-)\simeq\operatorname{Tot}(p^*L^*)\simeq\C_z\times\C_t
\]
and
\[
\widetilde M_\infty:=\widetilde M\setminus p_M^{-1}(C_+)\simeq\operatorname{Tot}(p^*L)\simeq\C_z\times\C_v,
\]
where $t=\frac{\beta}{\alpha}$ on $\widetilde M_0$ and $v=\frac{\alpha}{\beta}$ on $\widetilde M_\infty$. 
It follows that the $\Lambda$-actions on $\widetilde M_0$ and $\widetilde M_\infty$ are respectively given by
\[
\lambda\cdot(z,t)=(z+\lambda,b_\lambda(z)t), \qquad
\lambda\cdot(z,v)=(z+\lambda,b_\lambda(z)^{-1}v). 
\]\nobreak\hfill$\square$\par
} 
\end{notation}

We now consider the  $\Lambda$-action on $\C_z \times \C^2$  given by
\[
(m\tau+n)\cdot(z,a,b)=(z+m\tau+n,a+Kmb,b),
\]
and let $B_K$ be the vector bundle defined by
\[
\mathrm{Tot}(B_K)=(\C_z\times\C^2)/\Lambda.
\]
The map 
$$
\C_z \times \C^2 \rightarrow \C_z \times \C, \qquad (z, a, b) \ \mapsto \ (z, b)
$$
is $\Lambda$-equivariant, and since the $\Lambda$-action on $\C_z \times \C^2 $ is trivial on $b$, we have a well-defined morphism
$$
\mathrm{Tot}(B_K) \rightarrow \mathrm{Tot}(\sO_E).
$$
Let $\theta\in H^0(E,\Omega_E^1)$ be such that $p^*\theta=dz$, and define the morphism
$$
\iota_K \colon \mathrm{Tot}(\Omega_E) \to \mathrm{Tot}(B_K), \qquad (\bar z, a \theta) \ \mapsto \overline{(z, a, 0)}.
$$
The morphism $\iota_K$ defines an inclusion of locally free sheaves $\Omega_E^1 \xrightarrow{\iota_K} B_K$ (which we still denote by $\iota_K$), and putting everything together, we obtain a short exact sequence of locally free sheaves
\begin{equation}
\label{defineBK}
0\to \Omega_E^1 \xrightarrow{\iota_K} B_K \to \sO_E\to 0.
\end{equation}

\begin{lemma} \label{lemma:cechBK}
In the situation of Setup \ref{setup1}, 
let $\{V_i\}$ be an open cover of $E$ such that each $V_i\simeq \Delta$ and every non-empty intersection $V_i\cap V_j$ is connected. Choose holomorphic sections 
$$
s_i\colon V_i\to\C
$$ of $p$, and on every non-empty intersection $V_i\cap V_j$ write
\[
s_i-s_j=\lambda_{ij}=m_{ij}\tau+n_{ij}.
\]
Then the \v{C}ech $1$-cocycle $\{Km_{ij}\theta\}$ represents the extension class $\beta_\omega$ corresponding to \eqref{defineBK}.
\end{lemma}

\begin{proof}
For every $i$, we define a 
$$
\sigma_i\in H^0(V_i,B_K), \qquad \sigma_i(x):= \overline{(s_i(x),0,1)}.
$$ 
Then $\sigma_i$ is a local lift of $1\in H^0(V_i,\sO_E)$, so, by Fact \ref{fact-cech}, the differences $\sigma_j - \sigma_i$ form a \v Cech cocycle representing the extension class of $B_K$. By definition of $B_K$, since on $V_i\cap V_j$ we have
\[
\lambda_{ij}\cdot(s_j(x),-Km_{ij},1)=(s_i(x),0,1),
\]
it holds
\[
\overline{(s_i(x),0,1)}=\overline{(s_j(x),-Km_{ij},1)}.
\]
But then 
$$
\sigma_j(x)-\sigma_i(x)= \iota_K(Km_{ij}\theta)
$$
on $V_i\cap V_j$. 
\end{proof}

\begin{notation} \label{notation:C} {\rm
Let $V_i \subset E$ be the open sets defined in Lemma \ref{lemma:cechBK}. Since the line bundle $L$ trivialises on the disc $V_i \simeq \Delta$ we have $\fibre{f}{V_i} \simeq V_i \times \PP^1$ and we choose homogeneous coordinates 
$[\alpha_i:\beta_i]$ on the $\PP^1$-factor so that 
$$
C_-=\{\alpha_i=0\}, \qquad C_+=\{\beta_i=0\}.
$$ 
Set 
$$
M_{i,0}:=\{\alpha_i\neq0\}=f^{-1}(V_i)\setminus C_-, \qquad M_{i,\infty}:=\{\beta_i\neq0\}=f^{-1}(V_i)\setminus C_+,
$$ 
and observe that they are Stein, since isomorphic to $V_i\times \C$. Note that these notations are consistent with Notation \ref{notation:tildeM}, i.e. for all $i$ we have
$$
M_{i,0} \simeq s_i(V_i) \times \C_t \subset \tilde M_0,
\qquad
M_{i,\infty} \simeq s_i(V_i) \times \C_v \subset \tilde M_\infty.
$$
Since these sets are Stein, we can use this cover to write down a \v{C}ech representative
of the class $\xi_C = [\sO_M(-C_-)] \in H^1(M, \Omega^1_M)$: the curve $C_-$ is disjoint from 
$M_{i, 0}$, so that we can pick the constant polynomial $1$ on $M_{i, 0}$. On $M_{i,\infty}$, we have the affine coordinate $v_i:=\alpha_i/\beta_i$, and here $C_-$ is given by the polynomial $v_i$. This determines the transition function of the line bundle $\sO_M(-C_-) \in H^1(M, \sO_M^*)$ for this cover, and via
the map 
$$
d \log\colon \sO_M^* \rightarrow \Omega^1_M, \qquad f \mapsto \frac{df}{f}
$$
the corresponding cohomology class \cite[III,Ex.7.4]{Har77}. A short computation shows that the \v{C}ech representative of $\xi_C$ is 
\begin{eqnarray*}
0  & \mbox{on }  M_{i,0}\cap M_{j,0},
\\
d(\frac{v_j}{v_i})/\frac{v_j}{v_i} &  \mbox{on } M_{i,\infty}\cap M_{j,\infty}, 
\\
\frac{dv_j}{v_j} & \mbox{on } M_{i,0}\cap M_{j,\infty},
\\
\frac{-dv_i}{v_i} &  \mbox{on } M_{i,\infty}\cap M_{j,0}.
\end{eqnarray*}
}
\end{notation}

\begin{lemma} \label{lemma:cech-omega}
Let 
\begin{equation}
\label{dual:canonical-extension}
0\to\Omega_M^1\to W\to\sO_M\to0
\end{equation}
be the dual sequence of \eqref{eqn-canonicalext}. 
Then we can choose local lifts $\sigma_{i,0} \in H^0(M_{i,0}, W)$ and $\sigma_{i,\infty} \in H^0(M_{i,\infty}, W)$ of $1 \in H^0(M, \sO_M)$ such that, using the covering from Notation \ref{notation:C} and the notation of Lemma \ref{lemma:cechBK},
the holomorphic $1$-forms
\begin{eqnarray}
\label{imp1}
\sigma_{j,0}-\sigma_{i,0}=Km_{ij}f^*\theta & \mbox{on } M_{i,0}\cap M_{j,0}
\\
\label{imp2}
\sigma_{j,\infty}-\sigma_{i,\infty}=\rho\frac{d(v_j/v_i)}{v_j/v_i}+Km_{ij}f^*\theta
& \mbox{on }
M_{i,\infty}\cap M_{j,\infty}
\\
\label{imp3}
\sigma_{j,\infty}-\sigma_{i,0}=\rho\frac{dv_j}{v_j}+Km_{ij}f^*\theta
& \mbox{on }
M_{i,0}\cap M_{j,\infty}
\end{eqnarray}
define a \v Cech cocycle representing the class $[\omega]$.
\end{lemma}

\begin{proof}
Since the sets $M_{i,0}$ and $M_{i,\infty}$ are Stein, it is clear that local lifts $\hat \sigma_{i,0}, \hat \sigma_{i, \infty}$ of $1 \in H^0(M, \sO_M)$ exist on these sets. By Fact \ref{fact-cech}, the difference of these lifts defines a \v Cech cocycle representing the class $[\omega]=\rho \xi_C+f^*\beta_\omega$. By Notation \ref{notation:C} and Lemma \ref{lemma:cechBK}, the local forms in the statement define a \v Cech cocycle representing $[\omega]$. Since two cocycles representing the same class differ by a coboundary, we can use this coboundary to replace the local lifts by $\sigma_{i,0}, \sigma_{i, \infty}$ so that the formula holds.
\end{proof}

\begin{lemma} \label{lemma:coordinates}
Set
$$
\widetilde Z_0 := \fibre{\tilde \pi}{\widetilde M_0},\;
\widetilde Z_\infty := \fibre{\tilde \pi}{\widetilde M_\infty},
$$ 
where $\widetilde M_0 \simeq \C_z \times \C_t$ and $\widetilde M_\infty \simeq \C_z \times \C_v$
are defined in Notation \ref{notation:tildeM}. Then we have isomorphisms
$$
\widetilde Z_0\simeq\C_z\times\C_t\times\C_u\times\C_y,\; 
\widetilde Z_\infty\simeq\C_z\times\C_v\times\C_u\times\C_{y_\infty},
$$ 
and on $\widetilde Z_0\cap\widetilde Z_\infty = \{ p \in \C_z\times\C_t\times\C_u\times\C_y \ | \  t \neq 0 \}$ we have 
$$
v=\frac{1}{t},\;
y_\infty=-t^2y-\rho t.
$$ 
Moreover the $\Lambda$-action on $\widetilde Z_0\simeq\C_z\times\C_t\times\C_u\times\C_y$ is given
for $\lambda=m \tau +n$ by
$$
\lambda \cdot (z,t,u,y) =
\Bigg(z+\lambda, b_\lambda(z) t, u+m(K+2\pi i ty), \frac{y}{b_\lambda(z)}\Bigg).
$$
\end{lemma}

\begin{remark} \label{remark-total-space}
Let $\holom{p_M}{\widetilde M}{M}$ be the universal cover and
\[
0 \to p_M^* \Omega_M^1\to p_M^* W \xrightarrow{\epsilon} p_M^* \sO_M \to0
\]
the pull-back of the dual sequence of \eqref{eqn-canonicalext}. Then we have an isomorphism
$$
\tilde Z \simeq \mathbf P(p_M^* W) \setminus \mathbf P(p_M^* \Omega^1_M).
$$ 
Let $B \subset \tilde M$ be an open set such that there exists a local lift $\sigma \in \Gamma(B,p_M^*W)$ of $1 \in \Gamma(B, p_M^* \sO_M)$. Then every point of 
$\fibre{\tilde \pi}{B} \subset \tilde Z$ can be uniquely written as $\sigma+\eta$, with $\eta\in \Gamma(B, p_M^*\Omega_M^1)$. Thus we obtain a biholomorphism
$$
\fibre{\tilde \pi}{B} \simeq \mbox{Tot}(p_M^* \Omega^1_M|_B)
$$
\end{remark}

\begin{proof}[Proof of Lemma \ref{lemma:coordinates}]
By definition $p^*\theta=dz$, so pulling back \eqref{imp1} via $p_M\colon\widetilde M\to M$ 
we obtain
\[
p_M^*\sigma_{j,0}-p_M^*\sigma_{i,0}=Km_{ij}\, \tilde f^*dz
\]
on $p_M^{-1}(M_{i,0}\cap M_{j,0})$. Since $p^*\beta_\omega=0$ in $H^1(\C,\Omega_\C^1)$, the cocycle $\{Km_{ij}\,dz\}$ is a coboundary, thus there exist holomorphic $1$-forms $q_i$ on $p^{-1}(V_i)$ such that
\[
\tilde f^* q_j- \tilde f^* q_i=Km_{ij}\, \tilde f^* dz.
\]
 We now choose the $\{q_i\}$ so that the $\Lambda$-action on $\sigma_0$ can be explicitly computed. For $z\in p^{-1}(V_i)$ write $z=s_i(p(z))+\lambda_i(z)$, where $\lambda_i(z)=m_i(z)\tau+n_i(z)\in\Lambda$ and the functions $m_i$ and $n_i$ are locally constant. Following the notation of Lemma \ref{lemma:cechBK}, we have
$$
s_i-s_j=\lambda_{ij}=m_{ij}\tau+n_{ij},
$$
on $V_i\cap V_j$, and on $p^{-1}(V_i\cap V_j)$ it holds $\lambda_j-\lambda_i=\lambda_{ij}$, and therefore $m_j-m_i=m_{ij}$. Then if we set $q_i:=Km_i(z)\,dz$, we have $q_j-q_i=Km_{ij}dz$, and hence
\[
\tilde f^*q_j-\tilde f^*q_i=Km_{ij}\,\tilde f^*dz.
\]
Moreover, for $\lambda=m\tau+n\in\Lambda$, we have
\[
m_i(z+\lambda)=m_i(z)+m,
\]
and so
\[
q_i(z+\lambda)=q_i(z)+Km\,dz.
\] 
Then the local sections $p_M^*\sigma_{i,0}- \tilde f^*q_i$ glue to a section $\sigma_0$ of $p_M^*W$ on $\widetilde M_0$, which is a lift of $1$ (cf. Lemma \ref{lemma:cech-omega}).
By Remark \ref{remark-total-space} we obtain $\tilde Z_0 \simeq \mbox{Tot}(p_M^* \Omega^1_M|_{\tilde M_0})$. Since $\widetilde M_0 \simeq \C_z \times \C_t$ every point of $\widetilde Z_0$ can be written uniquely as
\begin{equation}
\label{coordinatetildeZ0}
\sigma_0+u\,dz+y\,dt,
\end{equation}
this gives the first isomorphism.
\medskip

 Similarly, pulling back \eqref{imp2} via $p_M\colon\widetilde M\to M$  we obtain
\[
p_M^*\sigma_{j,\infty}-p_M^*\sigma_{i,\infty}=\rho\frac{d(p_M^*(v_j/v_i))}{p_M^*(v_j/v_i)}+Km_{ij}\,dz.
\]
Recall that $\widetilde M_\infty \simeq \C_z \times \C_v$ and  $v$ vanishes with order $1$ along $p_M^{-1}(C_-)$. But also the section $p_M^*v_i$ vanishes with order $1$ along $p_M^{-1}(C_-)$, and hence we can write $p_M^*v_i=a_iv$, where $a_i$ is a nowhere-vanishing holomorphic function on $p_M^{-1}(M_{i,\infty})$. Then $p_M^*(v_j/v_i)=a_j/a_i$, and therefore the sections
\[
p_M^*\sigma_{i,\infty}-\tilde f^*q_i-\rho\frac{da_i}{a_i}
\]
glue to a section $\sigma_\infty$ of $p_M^*W$ on $\widetilde M_\infty$ mapping to $1$. 
Since $\widetilde M_\infty \simeq \C_z \times \C_v$ every point of $\widetilde Z_\infty$ can be uniquely written as
\[
\sigma_\infty+u_\infty\,dz+y_\infty\,dv;
\]
this gives the second isomorphism (we will justify the identification $u_\infty=u$ in the next step).
\medskip

On $p_M^{-1}(M_{i,0}\cap M_{i,\infty})$, the pull-back of \eqref{imp3}, with $j=i$, gives
\[
p_M^*\sigma_{i,\infty}-p_M^*\sigma_{i,0}=\rho\frac{d(p_M^*v_i)}{p_M^*v_i}.
\]
Since $p_M^*v_i=a_iv$, it follows from the definitions of $\sigma_0$ and $\sigma_\infty$ that
\[
\sigma_\infty-\sigma_0=\rho\frac{dv}{v}=-\rho\frac{dt}{t}.
\]
Since $v=1/t$ and $dv=-dt/t^2$, on the overlap $\widetilde M_0\cap \widetilde M_{\infty}$ we have
\begin{equation}\label{eqn:coefficients}
\sigma_\infty+u_\infty\,dz+y_\infty\,dv=\sigma_0-\rho\frac{dt}{t}+u_\infty\,dz-\frac{y_\infty}{t^2}\,dt.
\end{equation}
If we compare the coefficients of $dz$ and $dt$ above, with $u$ and $y$ in \eqref{coordinatetildeZ0}, we obtain
\[
u_\infty=u,\; y_\infty=-t^2y-\rho t.
\]
Thus the coordinate $u$ is the same on both charts, and the statement on the coordinate change follows. 
\medskip

For the last statement recall first that by \eqref{action:tildeM} the action on $\tilde M \simeq \C_z \times \PP^1$ is given by 
$$
\lambda\cdot(z,[x_0:x_1])=(z+\lambda,[x_0:b_\lambda(z)x_1]).
$$
Since $t=\frac{x_1}{x_0}$, the action on the first two coordinates is clear. Let $x=(z,t)$ and set
$$
x'=\lambda\cdot x=(z+\lambda,b_\lambda(z)t).
$$
By the explicit choice of the forms $q_i$ above, we have
$$
q_i(z+\lambda)=q_i(z)+Km\,dz.
$$
We now determine the $\Lambda$-action on the section $\sigma_0$. First of all, note that the induced action on the pull-back $p_M^*W$ is given by $\lambda\cdot(x,w)=(x',w)$. In particular,
\[
p_M^*\sigma_{i,0}(x)=\bigl(x,\sigma_{i,0}(p_M(x))\bigr),
\]
and since $p_M(x)=p_M(x')$, we have
\begin{align*}
\lambda\cdot p_M^*\sigma_{i,0}(x)&=\bigl(x',\sigma_{i,0}(p_M(x))\bigr)\\
                                 &=\bigl(x',\sigma_{i,0}(p_M(x'))\bigr)\\
                                 &=p_M^*\sigma_{i,0}(x').
\end{align*}

We now consider $\tilde f^*q_i$. Since $d(z)$ is invariant under translations, we have $$\lambda \cdot (\tilde f^*q_i(x))=(\lambda \cdot \tilde f^*q_i)(x')=Km_i(z)\,\tilde f^*dz|_{x'}.$$ On the other hand, if we evaluate $\tilde f^*q_i$ at $x'=x+\lambda$, we get
\[
\tilde f^*q_i(x')=Km_i(z+\lambda)\,\tilde f^*dz|_{x'},
\]
and since $\lambda=m\tau+n$ and $m_i(z+\lambda)=m_i(z)+m$, we have
\[
\tilde f^*q_i(x')=K(m_i(z)+m)\,\tilde f^*dz|_{x'},
\]
whence
\[
\lambda\cdot(\tilde f^*q_i(x))=\tilde f^*q_i(x')-Km\,\tilde f^*dz|_{x'}.
\]
Now recall that locally we have $\sigma_0=p_M^*\sigma_{i,0}-\tilde f^*q_i$. Then, using the preceding formulas, we get
\begin{align*}
\lambda\cdot(\sigma_0(x))&=\lambda\cdot (p_M^*\sigma_{i,0}(x))-\lambda\cdot(\tilde f^*q_i(x))\\
                         &=p_M^*\sigma_{i,0}(x')-\tilde f^*q_i(x')+Km\,\tilde f^*dz|_{x'}\\
                         &=\sigma_0(x')+Km\,\tilde f^*dz|_{x'}.
\end{align*}
Now let $P=\sigma_0(x)+u\,dz+y\,dt$ be a point of $\widetilde Z_0$, and write
\[
\lambda\cdot P=\sigma_0(x')+u'\,dz+y'\,dt.
\]
Since $\lambda\cdot(\sigma_0(x))=\sigma_0(x')+Km\,\tilde f^*dz|_{x'}$, if we act with the element $\lambda^{-1}$, we obtain
\[
\lambda^{-1}\cdot(\sigma_0(x'))=\sigma_0(x)-Km\,\tilde f^*dz|_x,
\]
and $\lambda^{-1}\cdot(dz|_{x'})=dz|_x$. But $\lambda\cdot (z,t)=(z+\lambda,b_\lambda(z)t)$, and thus
\[
\lambda^{-1}\cdot(dt|_{x'})=((\lambda^{-1})^{-1})^*(dt|_{x'})=d\bigl(b_\lambda(z)t\bigr)\big|_x=b_\lambda'(z)t\,dz|_x+b_\lambda(z)\,dt|_x.
\]
Therefore, if we act with the element $\lambda^{-1}$ on $\lambda\cdot P$, we get
\begin{align}\label{eqn:P}
P &=\sigma_0(x)-Km\,dz+u'\,dz +y'\bigl(b_\lambda'(z)t\,dz+b_\lambda(z)\,dt\bigr) \nonumber\\
  &=\sigma_0(x)+\bigl(u'-Km+y'b_\lambda'(z)t\bigr)\,dz+y'b_\lambda(z)\,dt.
\end{align}
Now, if we compare \eqref{eqn:P} and the original expression of $P$, we finally arrive at $y'=\frac{y}{b_\lambda(z)}$, and
\begin{equation}\label{eqn:y'}
u'-Km+y'b_\lambda'(z)t=u.
\end{equation}
If we substitute the expression of $y'$ in \ref{eqn:y'} we obtain
\[
u'=u+Km-\frac{b_\lambda'(z)}{b_\lambda(z)}ty,
\]
and since $\frac{b_\lambda'(z)}{b_\lambda(z)}=-2\pi i m$ (cf. Example \ref{example:automorphy}), equation \eqref{eqn:y'} simplifies to $u'=u+m(K+2\pi i ty)$, whence the $\Lambda$-action is
\[
\lambda\cdot(z,t,u,y) = \left(z+\lambda,b_\lambda(z)t,u+m(K+2\pi i ty),\frac{y}{b_\lambda(z)}\right).
\]
\end{proof}

\begin{proof}[Proof of Proposition \ref{prop:key}]
By Lemma \ref{lemma:coordinates} and Lemma \ref{lemma:coordinates-quadric} the maps
\begin{align*}
F_0\colon \widetilde Z_0 \simeq\C_z\times\C_t\times\C_u\times\C_y &\to \C_z\times\C_u\times Q_0 \\
 (z,t,u,y)&\mapsto\bigl(z,u,\psi_0(t,y)\bigr),
\end{align*}
and
\begin{align*}
F_\infty\colon\widetilde Z_\infty \simeq \C_z\times\C_v\times\C_u\times\C_{y_\infty} &\to \C_z\times\C_u\times Q_\infty\\
 (z,v,u,y_\infty)&\mapsto\bigl(z,u,\psi_\infty(v,y_\infty)\bigr)
\end{align*}
are isomorphisms. By Lemma \ref{lemma:coordinates} the coordinate change 
on $\widetilde Z_0\cap\widetilde Z_\infty$ is $v=\frac{1}{t}$ and $y_\infty=-t^2y-\rho t$, 
so using the formulas for $\psi_0$ and $\psi_\infty$ one checks that 
the maps $F_0$ and $F_\infty$ glue to a global biholomorphism
\[
F\colon\widetilde Z \rightarrow \C_z\times\C_u\times Q.
\]
For the description of the $\Lambda$-action, note first that the action is determined by
the action on the Zariski dense open subset $\tilde Z_0 \subset \tilde Z \simeq
\C_z\times\C_u\times Q$.
The action on $\tilde Z_0$ is given by the last statement in Lemma \ref{lemma:coordinates} and we are left to write down the $\Lambda$-action on $\C_z\times\C_u \times Q$ that makes
$F_0$ equivariant.
We have
\begin{align*}
F_0(\lambda \cdot (z,t,u,y)) &= F_0\Bigg(z+\lambda,b_\lambda(z)t,u+m(K+2\pi i ty),\frac{y}{b_\lambda(z)}\Bigg) \\
&= \Bigg(z+\lambda,u+m(K+2\pi i ty), \psi_0\Bigg(b_\lambda(z)t,\frac{y}{b_\lambda(z)}\Bigg)\Bigg).
\end{align*}
Using the formula for $\psi_0$ in Lemma \ref{lemma:coordinates-quadric} one computes
$$
\psi_0\Bigg(b_\lambda(z) t, \frac{y}{b_\lambda(z)}\Bigg)
= \big([1:b_\lambda(z) t], [y: b_\lambda(z) (t y+\rho)]\big).
$$
and
$$
F_0(z,t,u,y) =  \bigl(z,u,\psi_0(t,y)\bigr) = (z,u, [1:t], [y:ty+\rho])
$$
For an equivariant action we have $\lambda \cdot F_0(z,t,u,y)  = F_0(\lambda \cdot (z,t,u,y))$,
so the action on the first two coordinates of $\C_z\times\C_u \times Q$ is clear.
Comparing the two equations above, we see that in homogeneous coordinates the $\Lambda$-action on $Q$ is  given by
$$
\lambda \cdot ([x_0:x_1],[w_0:w_1]) = ([x_0:b_\lambda(z) x_1],[w_0: b_\lambda(z) w_1]).
$$
\end{proof}

\section{Proof of the main result}

With the Notation \ref{notationQ}, consider the holomorphic function
\[
\mu_{Q}\colon Q \rightarrow \C, \qquad ([x_0:x_1],[w_0:w_1]) \mapsto \frac{x_0w_1+x_1w_0}{x_0w_1-x_1w_0}.
\]
Then the function
$$
\tilde \mu:= \mu_Q \circ p_Q \colon \tilde Z \simeq \C_z\times\C_u\times Q \rightarrow \C,
$$
where $p_Q$ is the projection $\C_z\times\C_u\times Q \to Q$, is invariant under the $\Lambda$-action defined in Proposition \ref{prop:key}. Indeed $\tilde \mu$ depends on $([x_0:x_1],[w_0:w_1])$ only, and the action on these coordinates is
$$
\lambda\cdot ([x_0:x_1],[w_0:w_1]) = ([x_0:b_\lambda(z) x_1],[w_0: b_\lambda(z) w_1]).
$$
Thus $\tilde \mu$ descends to a holomorphic function 
$$\holom{\mu}{Z}{\C}.$$  
One easily checks that $\mu_Q$ is surjective, and hence so is $\mu$.

\begin{theorem} \label{thm:degree-one}
In the Setup \ref{setup1}, consider the holomorphic function
$
\holom{\mu}{Z}{\C}
$
defined above.
Set $c_0:=1-\frac{K}{\pi i \rho} \in \C$. Then the fibre $\fibre{\mu}{c_0} \subset Z$ is not Stein,  in particular $Z$ is not Stein.
\end{theorem}

\begin{proof}
Using the definition of $\psi_0$ and $\psi_\infty$ in Lemma \ref{lemma:coordinates-quadric}
it is straightforward to verify that 
$$
\mu_{Q}\bigl(\psi_0(t,y)\bigr)=1+\frac{2ty}{\rho},
\qquad
\mu_Q(\psi_\infty(v,y_\infty))=-1-\frac{2vy_\infty}{\rho}.
$$
Since $\tilde \mu=\mu_Q \circ p_Q$ it is clear that
$$
\fibre{\tilde \mu}{c_0} \simeq \C_z \times \C_u \times \fibre{\mu_Q}{c_0}.
$$
The equation $x_0w_1+x_1w_0=\big(1-\frac{K}{\pi i \rho}\big) (x_0w_1-x_1w_0)$
defines a smooth conic $C \subset \PP^1 \times \PP^1$
that intersects the conic $x_0w_1-x_1w_0=0$ transversally.
Thus 
$$
\fibre{\mu_Q}{c_0} \simeq C \setminus V(x_0w_1-x_1w_0) \simeq \C^*
$$ 
and in the chart given by $\psi_0$, its equation is 
$$
ty = -\frac{K}{2\pi i} \quad \Leftrightarrow \quad y =- \frac{K}{2 \pi i t}. 
$$
In particular $\fibre{\mu_Q}{c_0}$ is contained in the chart given by $\psi_0$, so we are left to describe the quotient of
$$
\C_z \times \C_u \times \C_t^* \simeq
\C_z \times \C_u \times \Bigg\{ ty = - \frac{K}{2\pi i} \Bigg\} \subset \C_z\times\C_u\times\C_t\times\C_y
$$
under the $\Lambda$-action given in Lemma \ref{lemma:coordinates}\footnote{Note that the order of the coordinates differs from the order in Lemma \ref{lemma:coordinates}.}. Plugging in $ty = - \frac{K}{2\pi i}$ we obtain that the action simplifies
to
$$
\lambda\cdot (z,u,t,y)=\Bigg(z+\lambda, u, b_\lambda(z) t, \frac{-K}{2 \pi i b_\lambda(z) t}\Bigg).
$$
The action on the coordinate $u$ is trivial, so 
$$
\fibre{\mu}{c_0} \simeq \C_u \times (\C_z \times \C_t^*)/\Lambda,
$$
where the $\Lambda$-action on $(\C_z \times \C_t^*)$ is given by $\lambda\cdot (z,t)=(z+\lambda, b_\lambda(z) t)$. Note that the action extends to $\C_z \times \C_t$ and 
by Example \ref{example:automorphy} we have
$$
(\C_z \times \C_t)/\Lambda \simeq \mbox{\rm Tot}(L^*).
$$
Thus we have
$$
(\C_z \times \C_t^*)/\Lambda \simeq \mbox{\rm Tot}(L^*) \setminus C_+
$$
where $C_+$ is the zero section in the total space. Thus $(\C_z \times \C_t^*)/\Lambda$ is not Stein by
Example \ref{example:complement}.
\end{proof}

We need an elementary lemma which we prove for the convenience of the reader:

\begin{lemma} \label{lemma:elliptic-easy}
Let $E=\C/\Lambda$ be an elliptic curve, and let $L$ be a line bundle of degree $-d<0$ on $E$.
Let $\omega$ be K\"ahler class on $M = \PP(\sO_E \oplus L)$. Then there exists an étale quotient
$$
\holom{g}{E}{E'}
$$
and a line bundle $L'$ of degree $-1$ such that $L \simeq g^* L'$. Let 
$$
\holom{g_M}{M \simeq \PP(g^*(\sO_{E'} \oplus L'))}{M' := \PP(\sO_{E'} \oplus L')}
$$
be the induced \'etale quotient, then there exists a K\"ahler class $\omega'$ on $M'$ such that $\omega \simeq g^* \omega'$.
In particular we have an induced \'etale cover
$$
\holom{g_Z}{Z_M^\omega}{Z_{M'}^{\omega'}}.
$$
\end{lemma}

\begin{proof} We write $\Lambda = \Z \oplus \Z \tau$ and set $\Lambda' := \Z \oplus \Z (\frac{1}{d}\tau)$.
Then $E' := \C/\Lambda'$ is an elliptic curve and the inclusion $\Lambda \subset \Lambda'$ induces an \'etale map
$$
\holom{g}{E}{E'}
$$
of degree $d$. The pull-back $\holom{g^*}{\pic^1(E') \simeq E'}{\pic^d(E) \simeq E}$ is not constant and therefore surjective. This shows the existence of $L'$. The pull-back
$$
\holom{g_M^*}{H^2(M', \C)}{H^2(M, \C)}
$$
is injective since $M'$ is compact K\"ahler. The two vector spaces have the same dimension, so 
we obtain $\omega'$ such that $\omega \simeq g^* \omega'$. The class $\omega'$ is K\"ahler since this property is invariant under finite maps. The statement  on the canonical extension is then clear: since $\omega \simeq g^* \omega'$,
the extension of vector bundles defined by $\omega$ is the pull-back of the extension of vector bundles defined by $\omega'$. 
\end{proof}

\begin{proof}[Proof of Theorem \ref{thm:not-Stein}]
We know by Lemma \ref{lemma:elliptic-easy} that
there exists an \'etale quotient 
$$
\holom{g_M}{M}{M'}
$$
such that $M' := \PP(\sO_{E'} \oplus L')$ with $L'$ of degree $-1$ and 
a K\"ahler class $\omega'$ on $M'$ such that
we have an \'etale cover
$$
\holom{g_Z}{Z_M^\omega}{Z_{M'}^{\omega'}}.
$$
By \cite[Lemma 2]{Nar62} it suffices to prove the statement for $Z_{M'}^{\omega'}$.
Now we conclude with  Theorem \ref{thm:degree-one}. 
\end{proof}


\newcommand{\etalchar}[1]{$^{#1}$}
\providecommand{\bysame}{\leavevmode\hbox to3em{\hrulefill}\thinspace}
\providecommand{\MR}{\relax\ifhmode\unskip\space\fi MR }
\providecommand{\MRhref}[2]{%
  \href{http://www.ams.org/mathscinet-getitem?mr=#1}{#2}
}
\providecommand{\href}[2]{#2}

\end{document}